\documentclass[11pt]{amsart}
\usepackage[T1]{fontenc}
\usepackage{amsmath,amssymb}
\usepackage{mathtools}
\usepackage{enumerate}
\usepackage{amsfonts} 
\usepackage[hidelinks]{hyperref}
\usepackage{mathrsfs}
\usepackage{amsthm}
\usepackage{eufrak}
\usepackage{commath}
\usepackage{stackengine}
\usepackage{tikz}
\usepackage{xcolor}
\usepackage{comment}
\usepackage[left=1.1in,right=1.1in]{geometry}

\newcommand{\N}{\mathbb{N}}

\newcommand{\R}{\mathbb{R}}

\newtheorem*{theorem*}{Theorem}

\newtheorem{theorem}{Theorem}[section]

\newtheorem{corollary}[theorem]{Corollary}
\newtheorem{lemma}[theorem]{Lemma}

\title[Growth beyond exponent $3/2$ for convexity and iterated sum sets]{Growth beyond exponent $3/2$ for convexity and iterated sum sets}
\author{Michalis Kokkinos}
\address{Michalis Kokkinos, Institute of Analysis and Number Theory, TU Graz, Austria}
\email{mike\_kokkinos@hotmail.com}

\author{Oliver Roche-Newton}
\address{Oliver Roche-Newton, Institute of Analysis and Number Theory, TU Graz, Austria}
\email{o.rochenewton@gmail.com}
\date{}

\begin{document}

\begin{abstract} We prove that the bound
\[
\max \{ |16A|,|16f(A)| \} \gg_m |A|^{\frac{3}{2}+\frac{1}{162}}
\]
holds for any polynomial $f$ with degree $m \geq 2$ and any finite $A \subset \mathbb R$. This shows that the classical Jarn\'{i}k obstruction to growth beyond exponent $3/2$, which occurs for general strictly convex functions, cannot occur for polynomial functions.

\end{abstract}
\maketitle
\section{Introduction}

An important generalisation of the sum-product phenomenon is the idea that strictly convex or concave functions disrupt additive structure. There are many concrete statements which exhibit this general principle: for instance, it was proven by Elekes, Nathanson and Ruzsa \cite{ENR} that the bound
\[
|A+A|^2|f(A)+f(A)|^2 \gg |A|^5
\]
holds for any finite $A \subset \mathbb R$ and any strictly convex or concave function $f$. In particular, it follows that
\[
\max \{|A+A|,|f(A)+f(A)| \} \gg |A|^{5/4},
\]
and taking $f(x)= \log x$ immediately recovers a sum-product estimate of Elekes \cite{El}. Moreover, the proofs in \cite{El} and \cite{ENR} are essentially the same, utilising a single application of a variant of the Szemer\'{e}di-Trotter Theorem. For more on convexity and sum sets including several of the most up-to-date bounds for the many variants of these problems, see the work of Bloom \cite{Bl25}, Schoen and Shkredov \cite{ScSh} and Stevens and Warren \cite{StWa22}.

One can naturally consider this problem with more variables. We write $kB$ for the $k$-fold sum set of a set $B$. In \cite{ENR}, it was also shown that
\[
\max \{|kA|, |kf(A)| \} \gg |A|^{\frac{3}{2}-2^{-k}},
\]
holds for any finite $A \subset \mathbb R$ and any strictly convex/concave $f$,
and so we get a growth exponent arbitrarily close to $3/2$ as $k$ grows. At the cost of introducing a minus sign, the cleaner estimate
\begin{equation} \label{twosumsonediff}
\max \{|A+A-A|, |f(A)+f(A)-f(A)| \} \gg |A|^{3/2}
\end{equation}
was proven by Bradshaw \cite{Br}, building on earlier work of Hanson, Roche-Newton and Rudnev \cite{HRR}, which in turn developed an elegant and elementary ``squeezing" argument which had appeared in \cite{RSSS}. The main result of a very recent paper of Cushman, Demeter and Wu \cite{CDW} implies a variant of \eqref{twosumsonediff} with only sums, although a small quantitative loss is incurred. They combine incidence theoretic and harmonic analysis techniques to prove a near optimal bound for the three-fold additive energy of a point set on a convex curve, which implies that the bound
\begin{equation} \label{threesums}
\max \{|A+A+A|, |f(A)+f(A)+f(A)| \} \gg |A|^{3/2-\epsilon}
\end{equation}
holds for all $\epsilon >0$.

The exponent $3/2$ is a natural barrier for both these elementary and incidence-theoretic techniques, and it turns out that there is a good reason for this. A classical construction of Jarn\'{i}k \cite{Ja} shows that there exists a set $P$ of $\Theta(n^2)$ points in $[n^3] \times [n^3]$ on a piecewise linear convex function $f$. This $f$ can then be perturbed to give a strictly convex function $g$ which passes through the points of $P$. If we then define $A:=\{x:(x,y)\in P\}$, it follows that $A$ is a set of $\Theta(n^2)$ integers in $[n^3]$, and thus
\[
|A+ A| \ll |A|^{3/2}.
\]
Similarly, by defining $g(A):=\{y:(x,y)\in P\}$, we have
\[
\max \{|A+A|, |g(A)+ g(A)| \} \ll |A|^{3/2}.
\]
The same remains true for arbitrarily many sums and differences. That is, the Jarn\'{i}k construction yields an example whereby
\[
\max \{|kA-\ell A|, |kg(A)- \ell g(A)| \} \ll_{k,\ell} |A|^{3/2}.
\]
This example differentiates this problem from the sum-product problem over the integers, where unbounded growth for iterated sums and products was proven by Bourgain and Chang \cite{BoCh}. A similar result over the reals now follows by combining the work of Mudgal \cite{Mu} with the Gowers, Green, Manners and Tao \cite{GGMT} results for the weak polynomial Freiman-Ruzsa conjecture. The classical two-variable form of the sum-product conjecture over the reals was recently refuted in a remarkable paper of Bloom, Sawin, Schildkraut and Zhelezov \cite{BSSZ}, but it would still be a huge surprise if the exponent $\frac{3}{2}+c$ was not admissible for sum-products over the reals.

In this paper, we are interested in the problem of when a growth exponent beyond $3/2$ can be attained for convexity and sum sets. More concretely, we would like to answer the following question: for which strictly convex/concave functions $f$ do there exist absolute constants $k \in \mathbb N$ and $c>0$ such that the bound
\begin{equation} \label{3/2+c}
\max \{ |kA|, |kf(A)| \} \gg |A|^{\frac{3}{2}+c} 
\end{equation}
holds for all $A \subset \mathbb R$? Some results in this direction were given by Roche-Newton \cite{Ro}, where it was shown that \eqref{3/2+c} holds (with $k=16$ and $c=\frac{1}{162}$) provided that $f$ satisfies an additional technical condition; see the forthcoming Theorem \ref{core} for the precise statement. It was also verified in \cite{Ro} that this condition is satisfied for $f(x)= \ln x$ and $f(x)=\ln(e^x+ \lambda)$, thus proving that \eqref{3/2+c} holds for these choices of $f$.

The only examples that we are aware of for which \eqref{3/2+c} does not hold are those arising from the Jarn\'{i}k construction described above. Given that there is quite some choice for how to define the function $g$ therein, it is not obvious what the conjectured classification should be for the functions $f$ which satisfy \eqref{3/2+c}, and we cannot currently give a full answer to this question. Nonetheless, building on the work of \cite{Ro}, we are able to prove such a result for a broad and natural algebraic class of convex functions. Our main result is the following.
\begin{theorem} \label{thm:mainpoly}
    Let $f \in \mathbb R[x]$ be a polynomial of degree $m \geq 2$ and let $A \subset \mathbb R$ be a finite set. Then
    \[
    \max \{ |16A|, |16f(A)| \} \gg_m |A|^{\frac{3}{2}+ \frac{1}{162}}.
    \]
\end{theorem}

\subsection*{Notation}
Given variable quantities $X$ and $Y$, we write $X\ll Y$, or $Y\gg X$, to denote that there exists an absolute constant $c>0$ such that $X\le cY$. We may further write $X\ll_t Y$ if the implied constant depends on some parameter $t$.

The notation $R_f$ is used for the range of a function $f.$ 

\section{Proofs of the main results}
The main tool for this paper is the following corrected and slightly modified form of \cite[Theorem 3.1]{Ro}.

\begin{theorem}\label{core}
    Let $I\subseteq\R$ be an interval of real numbers and let $m\in\N$. Suppose
that $A\subseteq I$ is a finite set of reals and that $f : I \rightarrow\R$ is a strictly convex or concave function. For $d\in(0,\infty)$, let $I_d=\{x\in I:x+d\in I\}$ and define the function $f_d:I_d\rightarrow\R$ by the formula $f_d(x) = f(x+d)-f(x)$. Suppose that, for all $d \in (0, \infty)$, the first four derivatives of $f_d^{-1}:R_{f_d}\rightarrow I_d$ have a combined total of at most $M$ zeroes. Then 
\[
    \max \{|16A|, |13f(A)| \} \gg_M |A|^{\frac{3}{2} +\frac{1}{162}}.
    \]
\end{theorem}
We note here that the statement given above is slightly different from that which appeared in \cite{Ro}, where it was instead stated that the first $3$ derivatives of $f_d^{-1}$ have a combined total of at most $M$ zeroes. This was due to an error in \cite{Ro} which we found when preparing this paper. Fortunately, all of the applications of Theorem \ref{core} which appear in \cite{Ro} remain valid, as we have checked that the fourth derivative condition is still verified for all cases considered there. A smaller technical change above from the statement of \cite{Ro} is that we impose less restrictions on the interval $I$. The proof from \cite{Ro} applies verbatim to this situation, with the function $f_d$ restricted to $I_d$, allowing for this slight generalisation of the statement.

It is worth remarking here that different variants of the result stated above are available to be taken from \cite{Ro}. For instance, there are statements with a slightly bigger exponent which involve some plus and minus signs. Furthermore, by incorporating inequality \eqref{threesums} into the arguments, one can obtain a slight quantitative strengthening to the statement of Theorem \ref{core} above. We do not pursue this here.

In order to apply Theorem \ref{core}, we need to make some calculations regarding derivatives of the inverse function $f_d^{-1}$. In many cases it is impossible to express the inverse of $f_d$ explicitly, such as when $f$ is a polynomial of degree at least six. In those cases we use the Inverse Rule. Since $f$ is a strictly convex or concave function, $f_d$ is strictly monotone and hence invertible; since $f_d'$ never vanishes, the Inverse Function Theorem gives that $(f_d^{-1})^{(1)}$ exists. The Inverse Rule therefore implies that, upon using the convention $f_d^{-1}(y)=:x,$ we can express the derivatives of $f_d^{-1}$ as
\begin{align} 
\big(f_d^{-1}\big)^{(1)}(y)&=\frac{1}{f_d^{(1)}(x)}, \label{d1} \\
    \big(f_d^{-1}\big)^{(2)}(y)&=\frac{-f_d^{(2)}(x)}{\big(f_d^{(1)}(x)\big)^3},  \label{d2}\\ 
    \big(f_d^{-1}\big)^{(3)}(y)&=\frac{-f_d^{(3)}(x)f_d^{(1)}(x)+3\big(f_d^{(2)}(x)\big)^2}{\big(f_d^{(1)}(x)\big)^5},\text{ and} \label{d3} \\
    \big(f_d^{-1}\big)^{(4)}(y)&=\frac{10f_d^{(1)}(x)f_d^{(2)}(x)f_d^{(3)}(x)-\big(f_d^{(1)}(x)\big)^2f_d^{(4)}(x)-15\big(f_d^{(2)}(x)\big)^3}{\big(f_d^{(1)}(x)\big)^7}. \label{d4}
    \end{align}

    These derivatives are all defined on $R_{f_d}$, since $f_d'(x)$ is either positive or negative, for all $x\in I$, depending on whether $f$ is strictly convex or strictly concave, respectively. Observe that the first derivative of $f_d^{-1}$ does not vanish anywhere, and it can henceforth be disregarded from our considerations. Our goal is to show that for the functions $f$ we are interested in, none of the remaining three derivatives of $f_d^{-1}$ have more than a constant number of zeroes.

    The proof of Theorem \ref{thm:mainpoly} is split into two parts, as the case of quadratic polynomials needs to be handled separately. Polynomials of degree $m \geq 3$ are handled in the following result.

\begin{theorem}\label{generalpoly}
    Let $A$ be a finite set of reals. Let $f:\R\rightarrow\R$ be a polynomial function given by $f(x)= a_mx^m+a_{m-1}
x^{m-1}+\dotsc+a_1x+a_0\in\R[x]$ of degree $m\ge3$, and with $a_m\neq0$. Then, 
  \[
    \max \{|16A|, |13f(A)| \} \gg_m |A|^{\frac{3}{2} +\frac{1}{162}}.
    \]
\end{theorem}
\begin{proof}
    We start by observing that the function $f_d$ defined as indicated above is a polynomial function of degree $m-1$. The zeroes of $(f_d^{-1})^{(i)}$ occur precisely when the numerators of the expressions \eqref{d2}, \eqref{d3} and \eqref{d4} are zero. These numerators are polynomials with degree at most $m^3$, and hence they have at most $m^3$ zeroes unless they are identically zero. Hence, the technical condition on $f$ of Theorem \ref{core} is satisfied if we prove that none of the numerators of the derivatives in \eqref{d2}, \eqref{d3}, \eqref{d4} can be identically zero. 
    
    We consider the highest order term in the numerator of each of these three derivatives. For \eqref{d2}, the numerator is a polynomial of degree $m-3 \geq 0$, as required (note that in the case $m=3$ this is a non-zero constant polynomial). 

    The highest order term for the numerator in \eqref{d3} is 
     \begin{equation*}
        a_m^2m^2(m-1)^2(m-2)d^2\Big[-(m-3)+3(m-2)\Big]x^{2m-6}.
    \end{equation*}
    This vanishes for all $x\in \R$ precisely when $m-3=3(m-2)$, that is when $m=3/2$, which cannot occur. Lastly, we proceed to check for the non-vanishing of the numerator of \eqref{d4}. Here, the leading coefficient is \begin{equation*}
        \begin{split}
            a_m^3m^3(m-1)^3(m-2)d^3\Big[10(m-2)(m-3)-(m-3)(m-4)-15(m-2)^2\Big].
        \end{split}
    \end{equation*}
    We deduce that the leading coefficient is not zero, since this requires $6m^2-17m+12=0,$ that is $m\in\{3/2,4/3\}$. We have checked that $f$ satisfies the derivative conditions of Theorem \ref{core}. 
    
    We note that we cannot yet apply Theorem \ref{core} directly since the function $f$ as defined above is not strictly convex or concave. However, since $f''$ has degree $m-2$, the real line can be partitioned into $t$ intervals $C_i\subseteq\R$, for some $t\le m-1$, on which $f$ is strictly convex or concave. Then $|A|=\sum_{i=1}^t|C_i\cap A|$ and, by the Pigeonhole Principle, there exist $j\in\{1,\dotsc,t\}$ for which $|C_j\cap A|\ge|A|/t\gg_m|A|.$ We can now apply Theorem \ref{core} with the set of real numbers chosen to be the large intersection $C_j\cap A$ and with the strictly convex or concave function $f\big|_{C_j}$, since the relevant sumsets of the subset are contained in those of $A.$
\end{proof}

It remains to deal with the case of Theorem \ref{thm:mainpoly} when $f$ is quadratic. Observe that we cannot apply Theorem \ref{core} directly since $f_d$ is a linear polynomial function whose inverse has second derivative already identically zero. However, we essentially bypass this issue by showing that the inverse of the quadratic polynomial does satisfy the technical condition.
\begin{lemma}\label{inversequad}
    Consider $f_\pm(x)=\frac{-b\pm\sqrt{b^2-4ac+4ax}}{2a}$, with $a>0,$ defined on the maximal interval $I\subseteq\R$ that ensures $f_{\pm}$ is real-valued and smooth, that is $I=\Big(c-\frac{b^2}{4a},\infty\Big)$.
    Let $A\subseteq I$ be finite. Then
\[
    \max \{|16A|, |13f_\pm(A)| \} \gg |A|^{\frac{3}{2} +\frac{1}{162}}.
    \]
\end{lemma}
\begin{proof}
    We want to apply Theorem \ref{core} with the strictly convex or concave function being either $f_+$ or $f_-$. We proceed by a manual check of the number of possible zeroes among the first four derivatives of the inverse of $(f_\pm)_d$, where $d\in(0,\infty),$ which we denote by $h$. This inverse is given explicitly by\begin{equation*}
        h(x)=\frac{a^2x^4+4acx^2-2adx^2-b^2x^2+d^2}{4ax^2},
    \end{equation*}whence\begin{equation*}
        \begin{split}
            &h^{(1)}(x)=\frac{a^2x^4-d^2}{2ax^3};\\
            &h^{(2)}(x)=\frac{a^2x^4+3d^2}{2ax^4};\\
            &h^{(3)}(x)=\frac{-6d^2}{ax^5};\text{ and}\\
            &h^{(4)}(x)=\frac{30d^2}{ax^6}.
        \end{split}
    \end{equation*}
    These derivatives have a combined total of at most $2$ zeroes, and so an application of Theorem \ref{core} completes the proof. 
\end{proof}
Observe that $|13B|\le|16B|$ for every finite nonempty $B$, by adding three copies of a fixed element. The next result deals with the case $m=2$ in Theorem \ref{thm:mainpoly}, and thus completes its proof. We use the notation $X^2$ to denote the set of squares of a real finite set $X.$
\begin{corollary}\label{cor:quadpoly}
    Consider $g:\R\rightarrow\R\,;\; x\mapsto ax^2+bx+c\in\R[x]$, with $a\neq0$, and let $X$ be a finite set of reals. We have,
    \begin{equation}\label{quadraticpoly}       
    \max \{|16g(X)|, |13X| \} \gg |X|^{\frac{3}{2} +\frac{1}{162}}.
    \end{equation}
    In particular,
    \begin{equation*}
        \max\{|16X^2|,|13X|\}\gg|X|^{\frac{3}{2}+\frac{1}{162}}
    \end{equation*}
\end{corollary}
\begin{proof}
    Replacing $g$ by $-g$ if necessary, we may assume $a>0$, since $|16(-g(X))|=|16g(X)|$. Towards proving the first statement, partition $X$ as $X=X_1\sqcup X_2$ where $X_1=X\cap(-\infty,-\frac{b}{2a}\big]$ and $X_2=X\setminus X_1$. If $|X_1|\ge|X|/2$, then take $A=g(X_1)$ and $f_-(x)=\frac{-b-\sqrt{b^2-4ac+4ax}}{2a}$ in Lemma \ref{inversequad}, otherwise take $A=g(X_2)$ and $f_+(x)=\frac{-b+\sqrt{b^2-4ac+4ax}}{2a}$. We observe that the partition of $X$ and the use of different branches of $f_\pm$ are necessary for $f_\pm(g(X))=X$ to occur, since for a given $x\in X$ we have $f_\pm(g(x))=\frac{-b\pm|2ax+b|}{2a}$. The second part follows by taking $a=1$ and $b=c=0$.
\end{proof}

\section{Further examples}

In this section, we show that the $3/2$ barrier for growth can be broken for some other natural families of convex functions. We begin by observing the following simple corollary of Theorem \ref{thm:mainpoly}.

\begin{corollary}
    Let $r\ge2$ be an integer and let $X$ be a finite set of positive reals. Then,
\[
    \max \{ |16X|, |16X^{1/r}| \} \gg_r |X|^{\frac{3}{2}+ \frac{1}{162}}.
    \]
\end{corollary}
\begin{proof}
    Apply Theorem \ref{thm:mainpoly} with $A=X^{1/r}$ and $f(x)=x^r.$
\end{proof}

\begin{theorem}
    Let $A \subset \mathbb R \setminus \{0 \}$ be a finite set. Let $f:\R\!\setminus\!\{0\}\rightarrow\R$ be defined by $f(x)= x^{-r}$, where $r$ is a positive integer. Then,
    \[
    \max \{|16A|, |13f(A)| \} \gg_r |A|^{\frac{3}{2} +\frac{1}{162}}.
    \]
\end{theorem}
\begin{proof}
    We want to apply Theorem \ref{core}. We proceed as in the proof of Theorem \ref{generalpoly}, by checking the numerators of the second, third and fourth derivative of $f_d^{-1}.$ Firstly, \begin{equation*}
        \begin{split}
            -f_d^{(2)}(x)=\frac{r(r+1)}{\big(x(x+d)\big)^{r+2}}\Big[(x+d)^{r+2}-x^{r+2}\Big],
        \end{split}
    \end{equation*}which is itself a rational function whose numerator is polynomial of degree $r+1$, and hence has finitely many possible zeroes (at most $r+1$). Secondly, we consider the numerator of the third derivative of $f_d^{-1}$, that is \begin{equation*}
            -f_d^{(3)}(x)f_d^{(1)}(x)+3\big(f_d^{(2)}(x)\big)^2=\frac{r^2(r+1)d^2}{\big(x(x+d)\big)^{2r+4}}\Big[(r+1)(r+2)(2r+3)\,x^{2r+2}+l.o.t.\Big],
    \end{equation*}and of the fourth derivative, \begin{equation*}
        \begin{split}
            10&f_d^{(1)}(x)f_d^{(2)}(x)f_d^{(3)}(x)-\big(f_d^{(1)}(x)\big)^2f_d^{(4)}(x)-15\big(f_d^{(2)}(x)\big)^3\\
            &=\frac{r^3(r+1)d^3}{\big(x(x+d)\big)^{3r+6}}\Big[(r+1)^2(r+2)(2r+3)(3r+4)\,x^{3r+3}+l.o.t.\Big].
        \end{split}
    \end{equation*}In both of these cases, the polynomial in the numerator is not identically zero, since $r$ is a positive integer. Lastly, as $0$ is not in $A$, observe that one of the two sign classes contains at least half of $A$. Therefore, take the finite set of reals in Theorem \ref{core} to be $A'=A\cap\R_+$ if most elements of $A$ are positive, or the set containing the additive inverses $A'=-(A\cap\R_-)$ otherwise, together with the strictly convex function $f\big|_{(0,\infty)}$. Taking $A'=-(A\cap\R_-)$ is possible, since $|16A'|=|16(A\cap\R_-)|$ and $|13f(A')|=|13(-1)^{-r}(A\cap\R_-)^{-r}|=|13(A\cap\R_-)^{-r}|.$
\end{proof}

\section*{Acknowledgments} The authors are partially supported by the Austrian Science Fund (FWF) project PAT2559123. We thank Brandon Hanson for asking some nice questions which inspired this research, and for helping us to understand the significance of the Jarn\'{i}k construction with respect to the threshold exponent of $3/2$.

\renewcommand{\bibname}{\Large References}
\bibliography{main}        
\bibliographystyle{plain}  

\end{document}